\documentclass{article}
\usepackage{amsfonts}
\usepackage{amsmath}

\newtheorem{theorem}{Theorem}

\newtheorem{corollary}[theorem]{Corollary}

\newtheorem{definition}[theorem]{Definition}

\newtheorem{lemma}[theorem]{Lemma}

\newtheorem{remark}[theorem]{Remark}

\newenvironment{proof}[1][Proof]{\noindent\textbf{#1.} }{\ \rule{0.5em}{0.5em}}

\begin{document}

\title{Exponential generating functions for Apostol-Dedekind and related sums}
\author{Yilmaz Simsek \\
Department of Mathematics, Faculty of Science,\\ University of Akdeniz,Antalya, Türkiye;\\ ysimsek@akdeniz.edu.tr;\\orcid:0000-0002-0611-7141}
\maketitle

\begin{abstract}
One of the primary objectives of this paper is to construct generating functions for Apostol-Dedekind sums. These exponential
generating functions greatly facilitate the discovery of new formulas and relations involving these sums. Other primary objectives of this paper are
to define Apostol-Dedekind sums of higher-order and other certain family of
finite sums, and to construct their generating functions. These generating
functions include the Bernoulli polynomials and numbers, the Euler
polynomials and numbers, the Array polynomials, the Stirling numbers, the
Carlitz polynomials with their partial derivative equations, and the Hardy
sums etc. Moreover, by using these functions, many novel formulas for the
Apostol-Dedekind sums of higher-order are derived.



\end{abstract}

\section{Introduction}

We have not encountered in literature on the study of the Apostol-Dedekind
sums using exponential generating functions, yet. Gessel \cite{Gessel}
defined generating functions for the generalized Dedekind sums by using
partial fraction expansions, the logarithmic derivatives of the
factorizations a rational function defined for every root of unity. By
lattice point enumeration of rational polytopesand finite Fourier series
methods, Beck and Robins \cite{BeckS} defined the Fourier-Dedekind sums. We 
\cite{SimsekJMAA} constructed generating functions for the $q$-Dedekind sums.

The motivation of this paper is to construct exponential generating
functions for the Apostol-Dedekind sums, $S_{p}(\ell _{1},\ell _{2})$,
defined by Apostol \cite{Apostol}-\cite{Apostol1}:%
\begin{eqnarray}
S_{p}(\ell _{1},\ell _{2}) &=&\sum_{m=1}^{\ell _{2}-1}\frac{m}{\ell _{2}}\,%
\overline{B}_{p}\left( \frac{\ell _{1}m}{\ell _{2}}\right)  \label{1} \\
&=&\sum_{m=1}^{\ell _{2}-1}\frac{m}{\ell _{2}}\,B_{p}\left( \frac{\ell _{1}m%
}{\ell _{2}}-\left[ \frac{\ell _{1}m}{\ell _{2}}\right] \right) ,  \notag
\end{eqnarray}%
where $\ell _{1}\in \mathbb{Z}$, $\ell _{2}\in \mathbb{N}$ with $\left( \ell
_{1},\ell _{2}\right) =1$; $p\in \mathbb{N}_{0}=\mathbb{N}\cup \left\{
0\right\} $; $\left[ \alpha \right] $ is the greatest integer $\leq \alpha $%
; $\overline{B}_{p}\left( \alpha \right) =B_{p}\left( \alpha -\left[ \alpha %
\right] \right) $ denotes the Bernoulli function and $B_{p}\left( \alpha
\right) $ denotes the Bernoulli polynomials, defined by 
\begin{equation}
g(w,\alpha )=\frac{w}{e^{w}-1}e^{w\alpha }=\sum_{p=0}^{\infty }B_{p}(\alpha )%
\frac{w^{p}}{p!}.  \label{Bp}
\end{equation}

Exponential generating function for $S_{p}(h,k)$ can be defined as follows:

\begin{definition}
Let $\ell _{1},\ell _{2}\in \mathbb{N}$ with $(\ell _{1},\ell _{2})=1$. We
define%
\begin{equation}
H(w;\left( \ell _{1},\ell _{2}\right) )=\sum_{p=0}^{\infty }S_{p}(\ell
_{1},\ell _{2})\frac{w^{p}}{p!}.  \label{2}
\end{equation}
\end{definition}

We prove that the function $H(w;\left( \ell _{1},\ell _{2}\right) )$ can be
given by the following theorem:

\begin{theorem}
\label{Th-3a} Let $\ell _{1}\in \mathbb{Z}$ and $\ell _{2}\in \mathbb{N}$
with $(\ell _{1},\ell _{2})=1$. Then%
\begin{equation}
H(w;\left( \ell _{1},\ell _{2}\right) )=\frac{1}{\ell _{2}}\sum_{m=1}^{\ell
_{2}-1}mg\left( w,\frac{\ell _{1}m}{\ell _{2}}-\left[ \frac{\ell _{1}m}{\ell
_{2}}\right] \right) ,  \label{3a}
\end{equation}
\end{theorem}

In the next sections, we show that there are many alternating generating
functions for the Apostol-Dedekind sums, $S_{p}(\ell _{1},\ell _{2})$. Our generating
functions can be written in terms of the Carlitz polynomials with their
partial derivative equations, the Euler numbers and polynomials, and the
Stirling numbers. These generating functions are also related to the , Hardy
sum's $s_{5}\left( \ell _{1},\ell _{2}\right) $. By using these generating
functions, we derive many novel formulas for the Apostol-Dedekind sums, $%
S_{p}(h,k)$.

We construct generating functions for Apostol-Dedekind sums of order $k$, denoted by $S_{p}^{(k)}(\ell _{1},\ell _{2}).$  By using equation (\ref{ah2}), we have the following results for these sums.

For $k=1$, $$S_{p}(\ell _{1},\ell _{2})=S_{p}^{(1)}(\ell _{1},\ell _{2}).$$

For $k=2$, $$S_{p}^{(2)}(\ell _{2},\ell _{2})=\sum_{j=0}^{p}\binom{p}{j}S_{j}(\ell _{1},\ell _{2})S_{p-j}(\ell _{1},\ell _{2}).$$

In the next sections, we show that there are many alternating generating
functions for the Apostol-Dedekind sums, $S_{p}(h,k)$. Our generating
functions can be written in terms of the Carlitz polynomials with their
partial derivative equations, the Euler numbers and polynomials, and the
Stirling numbers. These generating functions are also related to the , Hardy
sum's $s_{5}\left( \ell _{1},\ell _{2}\right) $. By using these generating
functions, we derive many novel formulas for the Apostol-Dedekind sums, $%
S_{p}(h,k)$.

We now briefly summarize our results as follows:

In Section 2, some basic definitions and notations are given.

In Section 3, we give generating functions for the Apostol-Dedekind $%
S_{p}(\ell _{1},\ell _{2})$ and sums $\mathcal{A}_{p}\left( \ell _{1},\ell
_{2}\right) $ which related to the Dedekind sums. We give proof of Theorem %
\ref{Th-3a}.

In Section 4, by the aid of these generating functions and partial
derivative of the Carlitz polynomials, we give some novel formulas for the
Apostol-Dedekind sums related to the Bernoulli polynomials and numbers, the
Euler polynomials and numbers, the Array polynomials, the Stirling numbers.

In Section 5, we construct generating functions not only for the
Apostol-Dedekind sum of order $k$, but also the sums $\mathcal{A}_{p}\left(
\ell _{1},\ell _{2}\right) $ of order $k$. Using these generating functions,
some novel formulas for the Apostol-Dedekind sum of order $k$ and the sums $%
\mathcal{A}_{p}\left( \ell _{1},\ell _{2}\right) $ of order $k$ are given.

\section{Preliminaries}

In this section, we give some basic definitions.

The generating function for the Bernoulli polynomials of order $k$ is given
by%
\begin{equation}
g^{k}(w,z)=\frac{w^{k}}{\left( e^{w}-1\right) ^{k}}e^{wz}=\sum_{p=0}^{\infty
}B_{p}^{(k)}(z)\frac{w^{p}}{p!},  \label{Bpk}
\end{equation}%
where $B_{p}(z)=B_{p}^{(1)}(z)$ and $B_{p}^{(k)}=B_{p}^{(k)}(0)$ (\textit{cf}%
. \cite{Milne}, \cite{Riardon}). The generating function for the Euler
polynomials is given by%
\begin{equation*}
h(w,z)=\frac{2}{e^{w}+1}e^{wz}=\sum_{p=0}^{\infty }E_{p}(z)\frac{w^{p}}{p!}
\end{equation*}

(\textit{cf}. \cite{Milne}, \cite{Riardon}). The generating function for the
Stirling numbers of the first kind is given by%
\begin{equation}
Fs1(w,k)=\ln ^{k}(1+w)=\sum_{p=0}^{\infty }\frac{k!}{p!}\left\langle
\begin{matrix}
	p\\
	k
\end{matrix}\right\rangle w^{p}  \label{Sn1}
\end{equation}
(\textit{cf}. \cite{Milne}, \cite{Riardon}). The generating function for the
Stirling numbers of the second kind is given by%
\begin{equation}
Fs2(w,k)=\left( e^{w}-1\right) ^{k}=\sum_{p=0}^{\infty }\frac{k!}{p!}\left\langle
\begin{matrix}
	p\\
	k
\end{matrix}\right\rangle w^{p}  \label{Sn}
\end{equation}
(\textit{cf}. \cite{Milne}, \cite{Riardon}). The generating function for the
array polynomials$\ $is given by%
\begin{equation}
\frac{\left( e^{w}-1\right) ^{v}e^{yw}}{v!}\sum_{a=0}^{\infty }S_{v}^{a}(y)%
\frac{w^{a}}{a!},  \label{Ap-1}
\end{equation}%
(\textit{cf}. \cite{Riardon}, \cite{SimsekFPTA}).

Some known properties of the Apostol-Dedekind sums, $S_{p}(\ell _{1},\ell
_{2})$ are given as follows,  see, for details (\textit{cf}. \cite{Apostol1,Apostol,Bayad,Beck,BeckArX,BeckS,BerndtDieter,BerndtGoldberg,gradimir,SimsekBOOK}): 

The reciprocity law of the Apostol-Dedekind sums is given by%
\begin{eqnarray*}
	\ell _{1}\ell _{2}{}^{p}S_{p}(\ell _{1},\ell _{2})+\ell _{2}\ell
	_{1}^{p}S_{p}(\ell _{2},\ell _{1}) 
	=\frac{p}{p+1}B_{p+1}+\frac{1}{p+1}\sum_{s=0}^{p+1}\binom{p+1}{s}%
	(-1)^{s}B_{s}B_{p+1-s}^{s}\ell _{1}^{s}\ell _{2}{}^{p+1-s}.
\end{eqnarray*}
When $p=0$ in (\ref{1}) we get%
\begin{equation}
	S_{0}(\ell _{1},\ell _{2})=\frac{1}{\ell _{2}}\sum_{m=1}^{\ell _{2}-1}m=%
	\frac{\ell _{2}-1}{2}.  \label{s0}
\end{equation}%
When $p=1$ in (\ref{1}), for $(\ell _{1},\ell _{2})=1$, we have the
classical Dedekind sum%
\begin{equation*}
	s(\ell _{1},\ell _{2}):=S_{1}(\ell _{1},\ell _{2}):
\end{equation*}%
\begin{equation}
	s(\ell _{1},\ell _{2})=\sum_{m=1}^{\ell _{2}-1}\frac{m}{\ell _{2}}\left( 
	\frac{\ell _{1}m}{\ell _{2}}-\left[ \frac{\ell _{1}m}{\ell _{2}}\right] -%
	\frac{1}{2}\right) ,  \label{s01}
\end{equation}%
where $\ell _{1}\in \mathbb{Z}$ and $\ell _{2}\in \mathbb{N}$ with $\left(
\ell _{1},\ell _{2}\right) =1$. Since%
\begin{equation*}
	\left( \left( x\right) \right) =\left\{ 
	\begin{array}{cc}
		x-\left[ x\right] -\frac{1}{2} & \text{if }x\notin \mathbb{Z} \\ 
		0 & \text{if }x\in \mathbb{Z},%
	\end{array}%
	\right.
\end{equation*}%
the $sum s(\ell _{1},\ell _{2})$ can also be given by
\begin{equation*}
	s(\ell _{1},\ell _{2})=\sum_{m=1}^{\ell _{2}-1}\left( \left( \frac{m}{\ell
		_{2}}\right) \right) \left( \left( \frac{m\ell _{1}}{\ell _{2}}\right).
	\right)
\end{equation*}

The reciprocity law of the Dedekind sums:

\begin{equation*}
	s\left( \ell _{1},\ell _{2}\right) +s\left( \ell _{2},\ell _{1}\right) =-%
	\frac{1}{4}+\frac{1}{12}\left( \frac{\ell _{1}}{\ell _{2}}+\frac{\ell _{2}}{%
		\ell _{1}}+\frac{1}{\ell _{1}\ell _{2}}\right) .
\end{equation*}

In this paper, we also need the following three and two-terms polynomial relations:
\begin{eqnarray*}
&&	(u-1)\sum_{x=1}^{a-1}u^{x-1}v^{\left[ \frac{bx}{a}\right] }w^{\left[ \frac{%
			cx}{a}\right] }+(v-1)\sum_{y=1}^{b-1}v^{y-1}w^{\left[ \frac{cy}{b}\right]
	}u^{\left[ \frac{ay}{b}\right] } +(w-1)\sum_{z=1}^{c-1}w^{z-1}u^{\left[ \frac{az}{c}\right] }v^{\left[ 
		\frac{bz}{c}\right] } \\
&=&u^{a-1}v^{b-1}w^{c-1}-1
\end{eqnarray*}%
and%
\begin{equation*}
	(u-1)\sum_{x=1}^{a-1}u^{x-1}v^{\left[ \frac{bx}{a}\right] }+(v-1)%
	\sum_{y=1}^{b-1}v^{y-1}u^{\left[ \frac{ay}{b}\right] }=u^{a-1}v^{b-1}-1
\end{equation*}%
(\textit{cf}. \cite{AV,BeckArX,BerndtDieter,Cetin,PS}).

The Carlitz polynomials are given by%
\begin{equation}
C(x_{1},x_{2},\ldots ,x_{n};\ell _{1},\ell _{2},\ldots ,\ell
_{n})=\sum_{v=1}^{\ell _{1}-1}x_{1}^{v-1}x_{2}^{\left[ \frac{v\ell _{2}}{%
\ell _{1}}\right] }\cdots x_{n}^{\left[ \frac{v\ell _{n}}{\ell _{1}}\right]
},  \label{cp}
\end{equation}%
where $x_{1},x_{2},\ldots ,x_{n}$ are indeterminates and $\ell _{1},\ell
_{2},\ldots ,\ell _{n}\in \mathbb{N}$, which are relatively prime in pairs (%
\textit{cf}. \cite{Beck}, see also \cite{Cetin}). $C(x_{1},x_{2};\ell
_{1},\ell _{2})$ known as the Dedekind--Carlitz polynomial (\textit{cf}. 
\cite{BeckArX}).

\section{Generating functions for the Apostol-Dedekind $S_{p}(\ell _{1},\ell
_{2})$}

In this section, we construct generating functions for the Apostol-Dedekind $%
S_{p}(\ell _{1},\ell _{2})$. In order to give these generating functions, we
start by the following lemmas:

\begin{lemma}
Let $\ell _{1},\ell _{2}\in \mathbb{N}$ with $(\ell _{1},\ell _{2})=1$. Let $%
x$ and $y$ are indeterminants. Then%
\begin{equation}
\sum_{k=1}^{\ell _{1}-1}kx^{k}y^{\left[ \frac{k\ell _{2}}{\ell _{1}}\right]
}=x^{2}\frac{\partial }{\partial x}\left\{ C(x,y;\ell _{1},\ell
_{2})\right\} +xC(x,y;\ell _{1},\ell _{2}).  \label{1cp}
\end{equation}
\end{lemma}

\begin{proof}
Using (\ref{cp}), we have%
\begin{equation}
C(x,y;\ell _{1},\ell _{2})=\sum_{k=1}^{\ell _{1}-1}x^{k-1}y^{\left[ \frac{%
k\ell _{2}}{\ell _{1}}\right] }.  \label{ah6}
\end{equation}%
Applying partial derivative operator $\frac{\partial }{\partial x}$ to the
polynomials $C(x,y;\ell _{1},\ell _{2})$ with respect to variable $x$, we get%
\begin{equation*}
\frac{\partial }{\partial x}\left\{ C(x,y;\ell _{1},\ell _{2})\right\}
=\sum_{k=1}^{\ell _{1}-1}\left( k-1\right) x^{k-2}y^{\left[ \frac{k\ell _{2}%
}{\ell _{1}}\right] }.
\end{equation*}%
Therefore,%
\begin{equation*}
x\frac{\partial }{\partial x}\left\{ C(x,y;\ell _{1},\ell _{2})\right\} =%
\frac{1}{x}\sum_{k=1}^{\ell _{1}-1}kx^{k}y^{\left[ \frac{k\ell _{2}}{\ell
_{1}}\right] }-\sum_{k=1}^{\ell _{1}-1}x^{k-1}y^{\left[ \frac{k\ell _{2}}{%
\ell _{1}}\right] },
\end{equation*}%
which gives the desired result.
\end{proof}

Some special values of Eq. (\ref{1cp}) are given as follows:

By using (\ref{1cp}), we define%
\begin{equation}
\sum_{k=1}^{\ell _{1}-1}kx^{k}y^{\left[ \frac{k\ell _{2}}{\ell _{1}}\right]
}=x^{2}P(x,y;\ell _{1},\ell _{2})+xC(x,y;\ell _{1},\ell _{2}),  \label{cp11}
\end{equation}%
where%
\begin{equation*}
P(x,y;\ell _{1},\ell _{2})=\frac{\partial }{\partial x}\left\{ C(x,y;\ell
_{1},\ell _{2})\right\} .
\end{equation*}%
Here we note that%
\begin{equation*}
P(\beta ,y;\ell _{1},\ell _{2})=\frac{\partial }{\partial x}\left\{
C(x,y;\ell _{1},\ell _{2})\right\} \left\vert _{x=\beta }\right. .
\end{equation*}%
Substituting $x=e^{\frac{\ell _{2}}{\ell _{1}}w}$ and $y=e^{-w}$ into (\ref%
{cp11}), we get

\begin{corollary}
Let $(\ell _{1},\ell _{2})=1$. Then%
\begin{equation}
\sum_{k=1}^{\ell _{1}-1}ke^{\left( \frac{k\ell _{2}}{\ell _{1}}-\left[ \frac{%
k\ell _{2}}{\ell _{1}}\right] \right) w}=e^{\frac{2\ell _{2}}{\ell _{1}}%
w}P\left( e^{\frac{\ell _{2}}{\ell _{1}}w},e^{-w};\ell _{1},\ell _{2}\right)
+e^{\frac{\ell _{2}}{\ell _{1}}w}C\left( e^{\frac{\ell _{2}}{\ell _{1}}%
w},e^{-w};\ell _{1},\ell _{2}\right) .  \label{cp1}
\end{equation}
\end{corollary}

Putting $x=-1$ and $y=-1$ into (\ref{cp11}), we get%
\begin{equation*}
\sum_{k=1}^{\ell _{1}-1}k(-1)^{k+\left[ \frac{k\ell _{2}}{\ell _{1}}\right]
}=P(-1,-1;\ell _{1},\ell _{2})-C(-1,-1;\ell _{1},\ell _{2}).
\end{equation*}%
Combining the above equation with the Hardy sums $s_{5}(\ell _{1},\ell _{2})$%
, defined by%
\begin{equation*}
\sum_{k=1}^{\ell _{1}-1}k(-1)^{k+\left[ \frac{k\ell _{2}}{\ell _{1}}\right]
}=\ell _{1}s_{5}(\ell _{1},\ell _{2})
\end{equation*}%
for $(\ell _{1},\ell _{2})=1$ (\textit{cf}. \cite{BerndtGoldberg}, \cite%
{SimsekJNT}), we get the following result:

\begin{corollary}
Let $(\ell _{1},\ell _{2})=1$. Then%
\begin{equation*}
s_{5}(\ell _{2},\ell _{1})=\frac{P(-1,-1;\ell _{1},\ell _{2})-C(-1,-1;\ell
_{1},\ell _{2})}{\ell _{1}}.
\end{equation*}
\end{corollary}

By applying the Cauchy-Euler operator $y\frac{\partial }{\partial y}$ to the
Carlitz polynomials $C(x,y;\ell _{1},\ell _{2})$ (given by Eq. (\ref{ah6})),
for $v,\ell \in \mathbb{N}_{0}$, we get%
\begin{eqnarray*}
Q_{v,\ell }(x,y;\ell _{1},\ell _{2}) &=&\left( x\frac{\partial }{\partial x}%
\right) ^{v}\left( y\frac{\partial }{\partial y}\right) ^{\ell }\left\{
C(x,y;\ell _{1},\ell _{2})\right\} \\
&=&\sum_{k=1}^{\ell _{1}-1}k^{v}\left[ \frac{k\ell _{2}}{\ell _{1}}\right]
^{\ell }x^{k}y^{\left[ \frac{k\ell _{2}}{\ell _{1}}\right] }.
\end{eqnarray*}%
We set%
\begin{equation*}
Q_{v,\ell }(x,y;\ell _{1},\ell _{2})=\left( x\frac{\partial }{\partial x}%
\right) ^{v}\left( y\frac{\partial }{\partial y}\right) ^{\ell }\left\{
C(x,y;\ell _{1},\ell _{2})\right\} .
\end{equation*}%
Some special values of the above polynomial are given by%
\begin{eqnarray*}
Q_{0,\ell }(1,1;\ell _{1},\ell _{2}) &=&\sum_{k=1}^{\ell _{1}-1}\left[ \frac{%
k\ell _{2}}{\ell _{1}}\right] ^{\ell }, \\
Q_{1,1}(1,1;\ell _{1},\ell _{2}) &=&\sum_{k=1}^{\ell _{1}-1}k\left[ \frac{%
k\ell _{2}}{\ell _{1}}\right] , \\
&&\vdots \\
Q_{v,\ell }(1,1;\ell _{1},\ell _{2}) &=&\sum_{k=1}^{\ell _{1}-1}k^{v}\left[ 
\frac{k\ell _{2}}{\ell _{1}}\right] ^{\ell }.
\end{eqnarray*}

We should emphasize that, with the help of the Carlitz polynomials, the
three and two-terms polynomial relations, it may be possible to calculate
the explicit computational formulas and values{}of the polynomial $Q_{v,\ell
}(x,y;\ell _{1},\ell _{2})$. However, this is the subject of another study
and are not be addressed here.

For instance,%
\begin{equation*}
Q_{0,\ell }(1,1;\ell _{1},\ell _{2})=\sum_{k=1}^{\ell _{1}-1}\left[ \frac{%
k\ell _{2}}{\ell _{1}}\right] ^{\ell }
\end{equation*}%
was studied by \cite{Broven}. He gave Faulhaber's formula and recurrence
relations for this sum.%
\begin{eqnarray*}
Q_{1,1}(1,1;\ell _{1},\ell _{2}) &=&\sum_{k=1}^{\ell _{1}-1}k\left[ \frac{%
k\ell _{2}}{\ell _{1}}\right] \\
&=&\ell _{1}s(\ell _{2},\ell _{1})+\frac{\ell _{2}(\ell _{1}-1)(2\ell _{1}-1)%
}{6\ell _{1}}+\frac{\ell _{1}-1}{4}
\end{eqnarray*}%
(\textit{cf}. \cite{BerndtDieter}, \cite{Mattew}, \cite{ascm2006}).

We set%
\begin{equation}
G(w,x,\ell _{2})=\frac{1}{\ell _{2}}\sum_{m=1}^{\ell _{2}-1}me^{\left( mx-%
\left[ mx\right] \right) w}=\sum_{p=0}^{\infty }\mathcal{J}_{p}\left( x,\ell
_{2}\right) \frac{w^{p}}{p!}  \label{Ypp}
\end{equation}%
Substituting $x=\frac{\ell _{1}}{\ell _{2}}$ with $\left( \ell _{1},\ell
_{2}\right) =1$ into (\ref{Ypp}), we get the following generating functions
for the Dedekind type sums:%
\begin{equation*}
F(w,\left( \ell _{1},\ell _{2}\right) ):=G\left( w,\frac{\ell _{1}}{\ell _{2}%
},\ell _{2}\right)
\end{equation*}%
and%
\begin{equation*}
\mathcal{A}_{p}\left( \ell _{1},\ell _{2}\right) :=\mathcal{J}_{p}\left( 
\frac{\ell _{1}}{\ell _{2}},\ell _{2}\right) .
\end{equation*}%
That is%
\begin{equation}
F(w,\left( \ell _{1},\ell _{2}\right) )=\frac{1}{\ell _{2}}\sum_{m=1}^{\ell
_{2}-1}me^{\left( \frac{\ell _{1}m}{\ell _{2}}-\left[ \frac{\ell _{1}m}{\ell
_{2}}\right] \right) w}=\sum_{p=0}^{\infty }\mathcal{A}_{p}\left( \ell
_{1},\ell _{2}\right) \frac{w^{p}}{p!},  \label{Yp}
\end{equation}%
where $\left( \ell _{1},\ell _{2}\right) =1$. By Eq. (\ref{Yp}), we obtain%
\begin{equation}
\mathcal{A}_{p}\left( \ell _{1},\ell _{2}\right) =\frac{1}{\ell _{2}}%
\sum_{m=1}^{\ell _{2}-1}m\left( \frac{\ell _{1}m}{\ell _{2}}-\left[ \frac{%
\ell _{1}m}{\ell _{2}}\right] \right) ^{p}.  \label{ah4a}
\end{equation}

By using (\ref{ah4a}), we get%
\begin{equation*}
\mathcal{A}_{p}\left( \ell _{1},\ell _{2}\right) =\frac{1}{\ell _{2}}%
\sum_{c=1}^{p}(-1)^{p-c}\binom{p}{c}\left( \frac{\ell _{1}}{\ell _{2}}%
\right) ^{p-c}\sum_{m=1}^{\ell _{2}-1}m^{p+1-c}\left[ \frac{\ell _{1}m}{\ell
_{2}}\right] ^{c}
\end{equation*}%
The sums $\mathcal{A}_{p}\left( \ell _{1},\ell _{2}\right) $ can be given in
terms of $Q_{v,\ell }(x,y;\ell _{1},\ell _{2})$ by the following corollary:

\begin{lemma}
Let $\ell _{1}\in \mathbb{Z}$, $p\in \mathbb{N}_{0}$, $\ell _{2}\in \mathbb{N%
}$ with $\left( \ell _{1},\ell _{2}\right) =1$. Then%
\begin{equation}
\mathcal{A}_{p}\left( \ell _{1},\ell _{2}\right) =\sum_{c=1}^{p}(-1)^{p-c}%
\binom{p}{c}\frac{\ell _{1}^{p-c}}{\ell _{2}^{p-c+1}}Q_{p+1-c,c}(1,1;\ell
_{1},\ell _{2}).  \label{ah7}
\end{equation}
\end{lemma}

The numerical values of the sums $\mathcal{A}_{p}\left( \ell _{1},\ell
_{2}\right) $\ can be also found by the Carlitz polynomials, the three and
two-terms polynomial relations. 
\begin{equation*}
\mathcal{A}_{p}\left( 2,3\right) =\frac{1}{3}\sum_{m=1}^{2}m\left( \frac{2m}{%
3}-\left[ \frac{2m}{3}\right] \right) ^{p}=\frac{2^{p}+2}{3^{p+1}}
\end{equation*}%
and%
\begin{equation*}
\mathcal{A}_{p}\left( 3,2\right) =\frac{1}{2}\sum_{m=1}^{1}m\left( \frac{3m}{%
2}-\left[ \frac{3m}{2}\right] \right) ^{p}=\frac{1}{2^{p}}.
\end{equation*}%
Putting $p=0,1,2$ in (\ref{ah4a}), we have%
\begin{equation*}
\mathcal{A}_{0}\left( \ell _{1},\ell _{2}\right) =\frac{\ell _{2}-1}{2}.
\end{equation*}%
\begin{equation*}
\mathcal{A}_{1}\left( \ell _{1},\ell _{2}\right) =\frac{\ell _{1}}{\ell
_{2}^{2}}\sum_{m=1}^{\ell _{2}-1}m^{2}-\frac{1}{\ell _{2}}\sum_{m=1}^{\ell
_{2}-1}m\left[ \frac{\ell _{1}m}{\ell _{2}}\right] .
\end{equation*}%
This equation gives%
\begin{equation*}
\sum_{m=1}^{\ell _{2}-1}m\left[ \frac{\ell _{1}m}{\ell _{2}}\right] =%
\mathcal{A}_{1}\left( \ell _{1},\ell _{2}\right) -\frac{\ell _{1}(\ell
_{2}-1)(2\ell _{2}-1)}{6}.
\end{equation*}%
Since%
\begin{equation*}
\left[ \frac{\ell _{1}m}{\ell _{2}}\right] =\frac{\ell _{1}m}{\ell _{2}}%
-\left( \left( \frac{m\ell _{1}}{\ell _{2}}\right) \right) -\frac{1}{2},
\end{equation*}%
we get%
\begin{eqnarray*}
\mathcal{A}_{1}\left( \ell _{1},\ell _{2}\right) &=&\frac{\ell _{1}}{\ell
_{2}^{2}}\sum_{m=1}^{\ell _{2}-1}m^{2}-\frac{1}{\ell _{2}}\sum_{m=1}^{\ell
_{2}-1}m\left[ \frac{\ell _{1}m}{\ell _{2}}\right] \\
&=&\frac{\ell _{1}(\ell _{2}-1)(2\ell _{2}-1)}{6\ell _{2}}-\frac{1}{\ell _{2}%
}\sum_{m=1}^{\ell _{2}-1}m\left( \frac{\ell _{1}m}{\ell _{2}}-\left( \left( 
\frac{m\ell _{1}}{\ell _{2}}\right) \right) -\frac{1}{2}\right) \\
&=&s\left( \ell _{1},\ell _{2}\right) +\frac{\ell _{2}-1}{4\ell _{2}}.
\end{eqnarray*}%
Thus, we get the following theorem:

\begin{theorem}
Let $\ell _{1}\in \mathbb{Z}$, $\ell _{2}\in \mathbb{N}$ with $\left( \ell
_{1},\ell _{2}\right) =1$. Then%
\begin{equation}
s\left( \ell _{1},\ell _{2}\right) =\mathcal{A}_{1}\left( \ell _{1},\ell
_{2}\right) +\frac{1-\ell _{2}}{4}.  \label{ah5}
\end{equation}
\end{theorem}

Putting $p=2$ in (\ref{ah4a}), we get%
\begin{eqnarray*}
\mathcal{A}_{p}\left( \ell _{1},\ell _{2}\right) &=&\frac{1}{\ell _{2}}%
\sum_{m=1}^{\ell _{2}-1}m\left( \frac{\ell _{1}^{2}m}{\ell _{2}^{2}}-2\frac{%
\ell _{1}m}{\ell _{2}}\left[ \frac{\ell _{1}m}{\ell _{2}}\right] +\left[ 
\frac{\ell _{1}m}{\ell _{2}}\right] ^{2}\right) \\
&=&\frac{\ell _{1}^{2}}{\ell _{2}^{3}}\sum_{m=1}^{\ell _{2}-1}m^{2}-2\frac{%
\ell _{1}}{\ell _{2}^{2}}\sum_{m=1}^{\ell _{2}-1}m^{2}\left[ \frac{\ell _{1}m%
}{\ell _{2}}\right] +\frac{1}{\ell _{2}}\sum_{m=1}^{\ell _{2}-1}m\left[ 
\frac{\ell _{1}m}{\ell _{2}}\right] ^{2}.
\end{eqnarray*}%
Combining the above equation with definition of $Q_{v,\ell }(x,y;\ell
_{1},\ell _{2})$, we get%
\begin{equation*}
\mathcal{A}_{p}\left( \ell _{1},\ell _{2}\right) =\frac{\ell _{1}^{2}(\ell
_{2}-1)(2\ell _{2}-1)}{6\ell _{2}^{2}}-\frac{2\ell _{1}}{\ell _{2}^{2}}%
Q_{2,1}(1,1;\ell _{1},\ell _{2})+\frac{1}{\ell _{2}}Q_{1,2}(1,1;\ell
_{1},\ell _{2})
\end{equation*}

Naturally, there are many studies in the literature that discuss the
relationships between among $\sum_{m=1}^{\ell _{2}-1}\left[ \frac{\ell _{1}m%
}{\ell _{2}}\right] $, $\sum_{m=1}^{\ell _{2}-1}\left[ \frac{\ell _{1}m}{%
\ell _{2}}\right] ^{k}$, $\sum_{m=1}^{\ell _{2}-1}m\left[ \frac{\ell _{1}m}{%
\ell _{2}}\right] $, $\sum_{m=1}^{\ell _{2}-1}m^{j}\left[ \frac{\ell _{1}m}{%
\ell _{2}}\right] ^{k}$ etc., the Dedekind--Carlitz polynomials, the
Dedekind type sums involving Dedekind-Rademacher Sums, Apostol-Dedekind
Sums. Some of these are (\textit{cf}. \cite{Beck}, \cite{BeckArX}, \cite%
{BerndtDieter}, \cite{Beck}, \cite{Broven}, \cite{Cetin}, \cite{RW} \cite{ascm2004}, 
\cite{ascm2006}).

We now give reciprocity law of the sums $\mathcal{A}_{1}\left( \ell
_{1},\ell _{2}\right) $ as follows:%
\begin{eqnarray*}
\mathcal{A}_{1}\left( \ell _{1},\ell _{2}\right) +\mathcal{A}_{1}\left( \ell
_{2},\ell _{1}\right) &=&s\left( \ell _{1},\ell _{2}\right) +s\left( \ell
_{2},\ell _{1}\right) +\frac{\ell _{1}-1}{4}+\frac{\ell _{2}-1}{4} \\
&=&-\frac{1}{4}\left( 3-\ell _{1}-\ell _{2}\right) +\frac{1}{12}\left( \frac{%
\ell _{1}}{\ell _{2}}+\frac{\ell _{2}}{\ell _{1}}+\frac{1}{\ell _{1}\ell _{2}%
}\right) .
\end{eqnarray*}%
Thus, a reciprocity law of the sum $\mathcal{A}_{1}\left( \ell _{1},\ell
_{2}\right) $ is given by%
\begin{equation*}
\mathcal{A}_{1}\left( \ell _{1},\ell _{2}\right) +\mathcal{A}_{1}\left( \ell
_{2},\ell _{1}\right) =\frac{\ell _{1}^{2}+\ell _{2}^{2}+\ell _{1}\ell
_{2}-3\left( 3-\ell _{1}-\ell _{2}\right) }{12\ell _{2}\ell _{1}}.
\end{equation*}%
By using (\ref{1}), we get the following result:

\begin{eqnarray}
S_{p}(\ell _{1},\ell _{2}) &=&\sum_{j=0}^{p}\binom{p}{j}B_{p-j}\sum_{m=1}^{%
\ell _{2}-1}\frac{m}{\ell _{2}}\left( \frac{\ell _{1}m}{\ell _{2}}-\left[ 
\frac{\ell _{1}m}{\ell _{2}}\right] \right) ^{j}  \label{1001} \\
&=&\sum_{j=0}^{p}\binom{p}{j}B_{p-j}\mathcal{A}_{j}\left( \ell _{1},\ell
_{2}\right) .  \notag
\end{eqnarray}

We now give some alternating generating functions for the Apostol-Dedekind
sums, $S_{p}(\ell _{1},\ell _{2})$.

\begin{theorem}
Let $\ell _{1}\in \mathbb{Z}$ and $\ell _{2}\in \mathbb{N}$ with $(\ell
_{1},\ell _{2})=1$. Then%
\begin{equation}
H(w;\left( \ell _{1},\ell _{2}\right) )=\frac{w}{\ell _{2}}\sum_{m=1}^{\ell
_{2}-1}m\frac{e^{\left( \frac{\ell _{1}m}{\ell _{2}}-\left[ \frac{\ell _{1}m%
}{\ell _{2}}\right] \right) w}}{e^{w}-1}.  \label{A}
\end{equation}
\end{theorem}

Combining (\ref{A}) with (\ref{Bp}), we also have the following alternative
form of the function $H(w;\left( \ell _{1},\ell _{2}\right) )$:%
\begin{equation}
H(w;\left( \ell _{1},\ell _{2}\right) )=g\left( w,0\right) F(w,\left( \ell
_{1},\ell _{2}\right) ),  \label{Aa}
\end{equation}%
and%
\begin{equation}
H(w;\left( \ell _{1},\ell _{2}\right) )=\frac{1}{\ell _{2}}\sum_{m=1}^{\ell
_{2}-1}mg\left( \frac{w}{2},\frac{\ell _{1}m}{\ell _{2}}-\left[ \frac{\ell
_{1}m}{\ell _{2}}\right] \right) h\left( \frac{w}{2},\frac{\ell _{1}m}{\ell
_{2}}-\left[ \frac{\ell _{1}m}{\ell _{2}}\right] \right) ,  \label{eg}
\end{equation}%
\begin{equation}
H(w;\left( \ell _{1},\ell _{2}\right) )=g\left( \frac{w}{2},0\right) h\left( 
\frac{w}{2},0\right) F(w,\left( \ell _{1},\ell _{2}\right) ),  \label{eq0}
\end{equation}%
\begin{equation}
H(w;\left( \ell _{1},\ell _{2}\right) )=\frac{1}{\ell _{2}}h\left( \frac{w}{2%
},0\right) \sum_{m=1}^{\ell _{2}-1}mg\left( \frac{w}{2},\frac{2\ell _{1}m}{%
\ell _{2}}-2\left[ \frac{\ell _{1}m}{\ell _{2}}\right] \right) ,  \label{eg1}
\end{equation}%
\begin{equation}
H(w;\left( \ell _{1},\ell _{2}\right) )=h\left( \frac{w}{2},0\right) H\left( 
\frac{w}{2};\left( \ell _{1},\ell _{2}\right) \right) ,  \label{eq2}
\end{equation}%
and%
\begin{equation}
H(w;\left( \ell _{1},\ell _{2}\right) )=\frac{1}{\ell _{2}}g\left( \frac{w}{2%
},0\right) \sum_{m=1}^{\ell _{2}-1}mh\left( \frac{w}{2},\frac{2\ell _{1}m}{%
\ell _{2}}-2\left[ \frac{\ell _{1}m}{\ell _{2}}\right] \right) .  \label{eq3}
\end{equation}

\begin{proof}[Proof of equation (\protect\ref{A})]
By combining (\ref{2}) with (\ref{1}), we get%
\begin{eqnarray*}
H(w;\left( \ell _{1},\ell _{2}\right) ) &=&\sum_{p=0}^{\infty }S_{p}(\ell
_{1},\ell _{2})\frac{w^{p}}{p!} \\
&=&\frac{1}{\ell _{2}}\sum_{p=0}^{\infty }\sum_{m=1}^{\ell _{2}-1}m\overline{%
B}_{p}\left( \frac{\ell _{1}m}{\ell _{2}}\right) \frac{t^{p}}{p!}.
\end{eqnarray*}%
Since%
\begin{equation*}
\overline{B}_{p}\left( \frac{\ell _{1}m}{\ell _{2}}\right) =B_{p}\left( 
\frac{\ell _{1}m}{\ell _{2}}-\left[ \frac{\ell _{1}m}{\ell _{2}}\right]
\right) ,
\end{equation*}%
we have%
\begin{equation*}
H(w;\left( \ell _{1},\ell _{2}\right) )=\frac{1}{\ell _{2}}%
\sum_{p=0}^{\infty }\sum_{m=1}^{\ell _{2}-1}mB_{p}\left( \frac{\ell _{1}m}{%
\ell _{2}}-\left[ \frac{\ell _{1}m}{\ell _{2}}\right] \right) \frac{t^{p}}{p!%
},
\end{equation*}%
which yields%
\begin{equation*}
H(w;\left( \ell _{1},\ell _{2}\right) )=\frac{1}{\ell _{2}}\sum_{m=1}^{\ell
_{2}-1}m\sum_{p=0}^{\infty }B_{p}\left( \frac{\ell _{1}m}{\ell _{2}}-\left[ 
\frac{\ell _{1}m}{\ell _{2}}\right] \right) \frac{t^{p}}{p!}.
\end{equation*}%
Combining the above equation with (\ref{Bp}), we get the desired result.
\end{proof}

\begin{proof}[Proof Theorem \protect\ref{Th-3a}]
Combining the above equation with (\ref{Bp}) and (\ref{A}), we get the
desired result.
\end{proof}

\begin{theorem}
Let $\ell _{1},\ell _{2}\in \mathbb{N}_{0}$. Let $(\ell _{1},\ell _{2})=1$.
Then%
\begin{equation*}
H(w;\left( \ell _{2},\ell _{1}\right) )=\frac{g(w,0)}{\ell _{1}}\left( e^{%
\frac{2\ell _{2}}{\ell _{1}}w}P\left( e^{\frac{\ell _{2}}{\ell _{1}}%
w},e^{-w};\ell _{1},\ell _{2}\right) +e^{\frac{\ell _{2}}{\ell _{1}}%
w}C\left( e^{\frac{\ell _{2}}{\ell _{1}}w},e^{-w};\ell _{1},\ell _{2}\right)
\right) .
\end{equation*}
\end{theorem}

\begin{proof}
By combining (\ref{Aa}) with (\ref{cp1}), after some calculations, we arrive
at the desired result.
\end{proof}

\section{Novel formulas the Apostol-Dedekind sums}

By using generating functions and their derivatives equations, we derive
many novel formulas for the Apostol-Dedekind sums related to the Bernoulli
polynomials and numbers, the Euler polynomials and numbers, the Array
polynomials, the Stirling numbers.

\begin{theorem}
Let $\ell _{1},\ell _{2},p\in \mathbb{N}_{0}$. Let $(\ell _{1},\ell _{2})=1$%
. Then%
\begin{eqnarray*}
S_{p}(\ell _{1},\ell _{2}) &=&\frac{p}{1-p}\sum_{\nu
=0}^{p-1}\sum_{j=0}^{p-1-v}\binom{p-1-v}{j}\binom{p-1}{\nu }B_{j}S_{\nu
}(\ell _{1},\ell _{2}) \\
&&+\frac{p}{1-p}\frac{1}{\ell _{2}}\sum_{m=1}^{\ell _{2}-1}m\left( \frac{%
\ell _{1}m}{\ell _{2}}-\left[ \frac{\ell _{1}m}{\ell _{2}}\right] \right) 
\overline{B}_{p-1}\left( \frac{\ell _{1}m}{\ell _{2}}\right) .
\end{eqnarray*}
\end{theorem}

\begin{proof}
Taking derivative of Eq. (\ref{3a}) with respect to $w$, we get%
\begin{equation*}
\frac{d}{dw}\left\{ H(w;\left( \ell _{1},\ell _{2}\right) )\right\} =\frac{1%
}{\ell _{2}}\frac{d}{dw}\left\{ \sum_{m=1}^{\ell _{2}-1}mg\left( w,\frac{%
\ell _{1}m}{\ell _{2}}-\left[ \frac{\ell _{1}m}{\ell _{2}}\right] \right)
\right\} .
\end{equation*}%
Thus, we obtain%
\begin{eqnarray*}
\frac{d}{dw}\left\{ H(w;\left( \ell _{1},\ell _{2}\right) )\right\} &=&\frac{%
1}{w}\frac{1}{\ell _{2}}\sum_{m=1}^{\ell _{2}-1}mg\left( w,\frac{\ell _{1}m}{%
\ell _{2}}-\left[ \frac{\ell _{1}m}{\ell _{2}}\right] \right) \\
&&-\frac{1}{\ell _{2}}\sum_{m=1}^{\ell _{2}-1}m\left( \frac{\ell _{1}m}{\ell
_{2}}-\left[ \frac{\ell _{1}m}{\ell _{2}}\right] \right) g\left( w,\frac{%
\ell _{1}m}{\ell _{2}}-\left[ \frac{\ell _{1}m}{\ell _{2}}\right] \right) \\
&&-\frac{e^{w}}{e^{w}-1}\frac{1}{\ell _{2}}\sum_{m=1}^{\ell _{2}-1}mg\left(
w,\frac{\ell _{1}m}{\ell _{2}}-\left[ \frac{\ell _{1}m}{\ell _{2}}\right]
\right) .
\end{eqnarray*}%
Therefore%
\begin{eqnarray*}
&&\frac{d}{dw}\left\{ H(w;\left( \ell _{1},\ell _{2}\right) )\right\}
-\left( \frac{1}{w}-g\left( w,1\right) \right) H(w;\left( \ell _{1},\ell
_{2}\right) ) \\
&=&\frac{1}{\ell _{2}}\sum_{m=1}^{\ell _{2}-1}m\left( \frac{\ell _{1}m}{\ell
_{2}}-\left[ \frac{\ell _{1}m}{\ell _{2}}\right] \right) g\left( w,\frac{%
\ell _{1}m}{\ell _{2}}-\left[ \frac{\ell _{1}m}{\ell _{2}}\right] \right) .
\end{eqnarray*}%
By using the above equation, we get%
\begin{eqnarray*}
&&\sum_{p=1}^{\infty }S_{p}(\ell _{1},\ell _{2})\frac{w^{p-1}}{(p-1)!}-\frac{%
1}{w}\sum_{p=0}^{\infty }S_{p}(\ell _{1},\ell _{2})\frac{w^{p}}{p!} \\
&=&-\sum_{p=0}^{\infty }\sum_{\nu =0}^{p}\binom{p}{\nu }B_{p-\nu }\left(
1\right) S_{\nu }(\ell _{1},\ell _{2})\frac{t^{p}}{p!} \\
&&-\frac{1}{\ell _{2}}\sum_{m=1}^{\ell _{2}-1}m\left[ \frac{\ell _{1}m}{\ell
_{2}}-\frac{\ell _{1}m}{\ell _{2}}\right] \sum_{p=0}^{\infty }\overline{B}%
_{p}\left( \frac{\ell _{1}m}{\ell _{2}}\right) \frac{w^{p}}{p!}.
\end{eqnarray*}%
Thus%
\begin{eqnarray*}
&&\sum_{p=0}^{\infty }S_{p+1}(\ell _{1},\ell _{2})\frac{w^{p+1}}{p!}%
-\sum_{p=0}^{\infty }S_{p}(\ell _{1},\ell _{2})\frac{w^{p}}{p!} \\
&=&\sum_{p=0}^{\infty }\sum_{\nu =0}^{p}\binom{p}{\nu }S_{\nu }(\ell
_{1},\ell _{2})B_{p-\nu }\left( 1\right) \frac{t^{p+1}}{p!} \\
&&-\frac{1}{\ell _{2}}\sum_{m=1}^{\ell _{2}-1}m\left( \frac{\ell _{1}m}{\ell
_{2}}-\left[ \frac{\ell _{1}m}{\ell _{2}}\right] \right) \sum_{p=0}^{\infty }%
\overline{B}_{p}\left( \frac{\ell _{1}m}{\ell _{2}}\right) \frac{w^{p+1}}{p!}%
,
\end{eqnarray*}%
which yields%
\begin{eqnarray*}
&&\sum_{p=0}^{\infty }pS_{p}(\ell _{1},\ell _{2})\frac{w^{p}}{p!}%
-\sum_{p=0}^{\infty }S_{p}(\ell _{1},\ell _{2})\frac{w^{p}}{p!} \\
&=&-\sum_{p=0}^{\infty }p\sum_{\nu =0}^{p-1}\binom{p-1}{\nu }S_{\nu }(\ell
_{1},\ell _{2})B_{p-1-\nu }\left( 1\right) \frac{t^{p}}{p!} \\
&&-\frac{1}{\ell _{2}}\sum_{m=1}^{\ell _{2}-1}m\left( \frac{\ell _{1}m}{\ell
_{2}}-\left[ \frac{\ell _{1}m}{\ell _{2}}\right] \right) \sum_{p=0}^{\infty
}p\overline{B}_{p-1}\left( \frac{\ell _{1}m}{\ell _{2}}\right) \frac{w^{p}}{%
p!}.
\end{eqnarray*}%
Comparing coefficient $\dfrac{w^{p}}{p!}$ on the both sides of the above
equation, we get%
\begin{eqnarray*}
\left( p-1\right) S_{p}(\ell _{1},\ell _{2}) &=&-p\sum_{\nu =0}^{p-1}\binom{%
p-1}{\nu }B_{p-1-\nu }\left( 1\right) S_{\nu }(\ell _{1},\ell _{2}) \\
&&-\frac{p}{\ell _{2}}\sum_{m=1}^{\ell _{2}-1}m\left( \frac{\ell _{1}m}{\ell
_{2}}-\left[ \frac{\ell _{1}m}{\ell _{2}}\right] \right) \overline{B}%
_{p-1}\left( \frac{\ell _{1}m}{\ell _{2}}\right) .
\end{eqnarray*}%
By the aid of the following known identity:%
\begin{equation*}
B_{p-\nu -1}\left( 1\right) =\sum_{j=0}^{p-v-1}\binom{p-v-1}{j}B_{j},
\end{equation*}%
we get the desired result.
\end{proof}

\begin{theorem}
Let $\ell _{1},\ell _{2},p\in \mathbb{N}_{0}$. Let $(\ell _{1},\ell _{2})=1$%
. Then%
\begin{equation*}
\sum_{j=0}^{p-1}\binom{p}{j}S_{j}(\ell _{1},\ell _{2})=\frac{p}{\ell _{2}}%
\sum_{m=1}^{\ell _{2}-1}m\left( \frac{\ell _{1}m}{\ell _{2}}-\left[ \frac{%
\ell _{1}m}{\ell _{2}}\right] \right) ^{p-1}.
\end{equation*}
\end{theorem}

\begin{proof}
By using (\ref{A}), we get%
\begin{equation*}
H(w;\left( \ell _{1},\ell _{2}\right) )\left( e^{w}-1\right) =\frac{w}{\ell
_{2}}\sum_{m=1}^{\ell _{2}-1}me^{\left( \frac{\ell _{1}m}{\ell _{2}}-\left[ 
\frac{\ell _{1}m}{\ell _{2}}\right] \right) w}.
\end{equation*}%
By combining (\ref{2}) with the above equation yields%
\begin{eqnarray*}
&&\sum_{p=0}^{\infty }\left( \sum_{j=0}^{p}\binom{p}{j}S_{j}(\ell _{1},\ell
_{2})-S_{p}(\ell _{1},\ell _{2})\right) \frac{w^{p}}{p!} \\
&=&\frac{1}{\ell _{2}}\sum_{m=1}^{\ell _{2}-1}m\sum_{p=0}^{\infty }p\left( 
\frac{\ell _{1}m}{\ell _{2}}-\left[ \frac{\ell _{1}m}{\ell _{2}}\right]
\right) ^{p-1}\frac{w^{p}}{p!}.
\end{eqnarray*}%
Comparing coefficient $\dfrac{w^{p}}{p!}$ on the both sides of the above
equation, we get%
\begin{equation*}
\sum_{j=0}^{p}\binom{p}{j}S_{j}(\ell _{1},\ell _{2})-S_{p}(\ell _{1},\ell
_{2})=\frac{p}{\ell _{2}}\sum_{m=1}^{\ell _{2}-1}m\left( \frac{\ell _{1}m}{%
\ell _{2}}-\left[ \frac{\ell _{1}m}{\ell _{2}}\right] \right) ^{p-1},
\end{equation*}%
which gives the desired result.
\end{proof}

\begin{theorem}
Let $\ell _{1},\ell _{2},p\in \mathbb{N}_{0}$. Let $(\ell _{1},\ell _{2})=1$%
. Then%
\begin{equation}
S_{p}(\ell _{1},\ell _{2})=\frac{1}{\ell _{2}}\sum_{m=1}^{\ell
_{2}-1}\sum_{j=0}^{p}\binom{p}{j}m\left( \frac{\ell _{1}m}{\ell _{2}}-\left[ 
\frac{\ell _{1}m}{\ell _{2}}\right] \right) ^{j}B_{p-j}.  \label{o1}
\end{equation}
\end{theorem}

\begin{proof}
By using (\ref{Aa}) and (\ref{Bp}), we get%
\begin{equation*}
\sum_{p=0}^{\infty }S_{p}(\ell _{1},\ell _{2})\frac{w^{p}}{p!}=\frac{1}{\ell
_{2}}\sum_{p=0}^{\infty }B_{p}\frac{t^{p}}{p!}\sum_{m=1}^{\ell
_{2}-1}m\sum_{p=0}^{\infty }\left( \frac{\ell _{1}m}{\ell _{2}}-\left[ \frac{%
\ell _{1}m}{\ell _{2}}\right] \right) ^{p}\frac{w^{p}}{p!}.
\end{equation*}%
By using Cauchy product on the RHS of the above function yields%
\begin{equation*}
\sum_{p=0}^{\infty }S_{p}(\ell _{1},\ell _{2})\frac{w^{p}}{p!}=\frac{1}{\ell
_{2}}\sum_{m=1}^{\ell _{2}-1}m\sum_{p=0}^{\infty }\sum_{j=0}^{p}\binom{p}{j}%
\left( \frac{\ell _{1}m}{\ell _{2}}-\left[ \frac{\ell _{1}m}{\ell _{2}}%
\right] \right) ^{j}B_{p-j}\frac{w^{p}}{p!}.
\end{equation*}%
Comparing coefficient $\dfrac{w^{p}}{p!}$ on the both sides of the above
equation, we get the desired result.
\end{proof}

\begin{theorem}
Let $\ell _{1},\ell _{2},p\in \mathbb{N}_{0}$. Let $(\ell _{1},\ell _{2})=1$%
. Then%
\begin{equation*}
S_{p}(\ell _{1},\ell _{2})=-\frac{p}{\ell _{2}}\sum_{m=1}^{\ell
_{2}-1}m\zeta \left( 1-p,\frac{\ell _{1}m}{\ell _{2}}-\left[ \frac{\ell _{1}m%
}{\ell _{2}}\right] \right) .
\end{equation*}
\end{theorem}

\begin{proof}
By applying binomial theorem to (\ref{A}) for $|e^{t}|<1,$ we get%
\begin{equation*}
\sum_{p=0}^{\infty }S_{p}(\ell _{1},\ell _{2})\frac{w^{p}}{p!}=-\frac{w}{%
\ell _{2}}\sum_{m=1}^{\ell _{2}-1}m\sum_{\nu =0}^{\infty }e^{\left( \frac{%
\ell _{1}m}{\ell _{2}}-\left[ \frac{\ell _{1}m}{\ell _{2}}\right] +\nu
\right) w}.
\end{equation*}%
Therefore%
\begin{equation*}
\sum_{p=0}^{\infty }S_{p}(\ell _{1},\ell _{2})\frac{w^{p}}{p!}=-\frac{1}{%
\ell _{2}}\sum_{m=1}^{\ell _{2}-1}m\sum_{\nu =0}^{\infty }\sum_{p=0}^{\infty
}p\left( \frac{\ell _{1}m}{\ell _{2}}-\left[ \frac{\ell _{1}m}{\ell _{2}}%
\right] +\nu \right) ^{p-1}\frac{w^{p}}{p!}.
\end{equation*}%
Combining the above function with the Hurwitz zeta function yields%
\begin{equation*}
\sum_{p=0}^{\infty }S_{p}(\ell _{1},\ell _{2})\frac{w^{p}}{p!}=-\frac{1}{%
\ell _{2}}\sum_{m=1}^{\ell _{2}-1}m\sum_{p=0}^{\infty }p\zeta \left( 1-p,%
\frac{\ell _{1}m}{\ell _{2}}-\left[ \frac{\ell _{1}m}{\ell _{2}}\right]
\right) \frac{w^{p}}{p!}.
\end{equation*}%
Comparing coefficient $\dfrac{w^{p}}{p!}$ on the both sides of the above
equation, we get the desired result.
\end{proof}

\begin{theorem}
Let $\ell _{1},\ell _{2},p\in \mathbb{N}_{0}$. Let $(\ell _{1},\ell _{2})=1$%
. Then%
\begin{equation}
S_{p}(\ell _{1},\ell _{2})=\frac{1}{\ell _{2}}\sum_{m=1}^{\ell
_{2}-1}m\sum_{\nu =0}^{p}\frac{(-1)^{v}v!}{v+1}S_{v}^{p}\left( \frac{\ell
_{1}m}{\ell _{2}}-\left[ \frac{\ell _{1}m}{\ell _{2}}\right] \right) .
\label{s1}
\end{equation}
\end{theorem}

\begin{proof}
Assuming that $|e^{w}-1|<1$, then Eq. (\ref{A}) reduces to%
\begin{equation}
H(w;\left( \ell _{1},\ell _{2}\right) )=\frac{1}{\ell _{2}}\sum_{m=1}^{\ell
_{2}-1}m\sum_{\nu =0}^{\infty }\frac{(-1)^{v}\left( e^{w}-1\right)
^{v}e^{\left( \frac{\ell _{1}m}{\ell _{2}}-\left[ \frac{\ell _{1}m}{\ell _{2}%
}\right] \right) w}}{v+1}.  \label{hr}
\end{equation}%
Combining the above function with (\ref{Ap-1}), we get%
\begin{equation*}
H(w;\left( \ell _{1},\ell _{2}\right) )=\frac{1}{\ell _{2}}\sum_{m=1}^{\ell
_{2}-1}m\sum_{\nu =0}^{\infty }\frac{(-1)^{v}v!}{v+1}\sum_{p=0}^{\infty
}S_{v}^{p}\left( \frac{\ell _{1}m}{\ell _{2}}-\left[ \frac{\ell _{1}m}{\ell
_{2}}\right] \right) \frac{w^{p}}{p!}.
\end{equation*}%
Therefore%
\begin{equation*}
\sum_{p=0}^{\infty }S_{p}(\ell _{1},\ell _{2})\frac{w^{p}}{p!}=\frac{1}{\ell
_{2}}\sum_{m=1}^{\ell _{2}-1}m\sum_{\nu =0}^{\infty }\frac{(-1)^{v}v!}{v+1}%
\sum_{p=0}^{\infty }S_{v}^{p}\left( \frac{\ell _{1}m}{\ell _{2}}-\left[ 
\frac{\ell _{1}m}{\ell _{2}}\right] \right) \frac{w^{p}}{p!}.
\end{equation*}%
Comparing coefficient $\dfrac{w^{p}}{p!}$ on the both sides of the above
equation, we get%
\begin{equation*}
S_{p}(\ell _{1},\ell _{2})=\frac{1}{\ell _{2}}\sum_{m=1}^{\ell
_{2}-1}m\sum_{p=0}^{\infty }\frac{(-1)^{v}v!}{v+1}S_{v}^{p}\left( \frac{\ell
_{1}m}{\ell _{2}}-\left[ \frac{\ell _{1}m}{\ell _{2}}\right] \right) .
\end{equation*}%
When $v>p$, $S_{v}^{p}\left( \frac{\ell _{1}m}{\ell _{2}}-\left[ \frac{\ell
_{1}m}{\ell _{2}}\right] \right) =0$, the above equation reduces to the
desired result.
\end{proof}

\begin{corollary}
Let $\ell _{1},\ell _{2},p\in \mathbb{N}_{0}$. Let $(\ell _{1},\ell _{2})=1$%
. Then%
\begin{equation*}
S_{p}(\ell _{1},\ell _{2})=\frac{1}{\ell _{2}}\sum_{m=1}^{\ell
_{2}-1}m\sum_{\nu =0}^{p}\sum_{j=0}^{v}\left\langle
\begin{matrix}
v\\
j
\end{matrix}\right\rangle B_{j}S_{v}^{p}\left( \frac{\ell _{1}m}{\ell _{2}}-\left[ 
\frac{\ell _{1}m}{\ell _{2}}\right] \right) .
\end{equation*}
\end{corollary}

\begin{proof}
Combining Eq. (\ref{s1}) with the following known formula, given in \cite[%
p.117]{Riardon},%
\begin{equation}
	\frac{(-1)^{v}v!}{v+1}=\sum_{j=0}^{v}\left\langle \begin{matrix}
		v\\
		j
	\end{matrix}\right\rangle B_{j},  \label{bs-11}
\end{equation}
we get the desired result.
\end{proof}

\begin{theorem}
Let $\ell _{1},\ell _{2},p\in \mathbb{N}_{0}$. Let $(\ell _{1},\ell _{2})=1$%
. Then%
\begin{eqnarray}
S_{p}(\ell _{1},\ell _{2}) &=&\frac{1}{\ell _{2}}\sum_{m=1}^{\ell
_{2}-1}m\sum_{c=0}^{p}\binom{p}{c}  \label{sDs} \\
&&\times \sum_{\nu =0}^{c}\frac{(-1)^{v}v!}{v+1}\left\langle
\begin{matrix}
	c\\
v
\end{matrix}\right\rangle \left( \frac{\ell _{1}m}{\ell _{2}}-\left[ \frac{\ell _{1}m}{%
\ell _{2}}\right] \right) ^{p-c}.  \notag
\end{eqnarray}
\end{theorem}

\begin{proof}
Combining (\ref{hr}) with (\ref{Sn}), we get the following functional
equation:%
\begin{equation*}
H(w;\left( \ell _{1},\ell _{2}\right) )=\frac{1}{\ell _{2}}\sum_{m=1}^{\ell
_{2}-1}m\sum_{\nu =0}^{\infty }\frac{(-1)^{v}v!}{v+1}e^{\left( \frac{\ell
_{1}m}{\ell _{2}}-\left[ \frac{\ell _{1}m}{\ell _{2}}\right] \right)
w}Fs2(w,v).
\end{equation*}%
Using this equation, we obtain%
\begin{eqnarray*}
\sum_{p=0}^{\infty }S_{p}(\ell _{1},\ell _{2})\frac{w^{p}}{p!} &=&\frac{1}{%
\ell _{2}}\sum_{m=1}^{\ell _{2}-1}m\sum_{\nu =0}^{\infty }\frac{(-1)^{v}v!}{%
v+1} \\
&&\times \sum_{p=0}^{\infty }\left\langle
\begin{matrix}
	p\\
v
\end{matrix}\right\rangle\frac{w^{p}}{%
p!}\sum_{c=0}^{\infty }\left( \frac{\ell _{1}m}{\ell _{2}}-\left[ \frac{\ell
_{1}m}{\ell _{2}}\right] \right) ^{c}\frac{w^{c}}{c!}.
\end{eqnarray*}%
With the aid of the Cauchy product in RHS, we get%
\begin{eqnarray*}
\sum_{p=0}^{\infty }S_{p}(\ell _{1},\ell _{2})\frac{w^{p}}{p!} &=&\frac{1}{%
\ell _{2}}\sum_{m=1}^{\ell _{2}-1}m\sum_{\nu =0}^{\infty }\frac{(-1)^{v}v!}{%
v+1} \\
&&\times \sum_{p=0}^{\infty }\sum_{c=0}^{p}\binom{p}{c}\left\langle
\begin{matrix}
	p-c\\
	v
\end{matrix}\right\rangle \left( \frac{\ell _{1}m}{\ell _{2}}-\left[ \frac{\ell _{1}m%
}{\ell _{2}}\right] \right) ^{c}\frac{w^{c}}{c!}.
\end{eqnarray*}%
Comparing coefficient $\dfrac{w^{p}}{p!}$ on the both sides of the above
equation, we get%
\begin{eqnarray*}
\sum_{p=0}^{\infty }S_{p}(\ell _{1},\ell _{2})\frac{w^{p}}{p!} &=&\frac{1}{%
\ell _{2}}\sum_{m=1}^{\ell _{2}-1}m\sum_{\nu =0}^{\infty }\frac{(-1)^{v}v!}{%
v+1} \\
&&\times \sum_{p=0}^{\infty }\sum_{c=0}^{p}\binom{p}{c}\left\langle
\begin{matrix}
	c\\
	v
\end{matrix}\right\rangle \left( \frac{\ell _{1}m}{\ell _{2}}-\left[ \frac{\ell _{1}m}{%
\ell _{2}}\right] \right) ^{p-c}\frac{w^{c}}{c!}.
\end{eqnarray*}%
When $v>c$, $\left\langle
\begin{matrix}
	c\\
	v
\end{matrix}\right\rangle =0$, the above equation
reduces to the desired result.
\end{proof}

\begin{remark}
Substituting (\ref{bs-11}) into (\ref{sDs}), we get%
	\begin{eqnarray*}
	S_{p}(\ell _{1},\ell _{2}) =\frac{1}{\ell _{2}}\sum_{m=1}^{\ell
		_{2}-1}m\sum_{c=0}^{p}\binom{p}{c}\left( \frac{\ell _{1}m}{\ell _{2}}-\left[ 
	\frac{\ell _{1}m}{\ell _{2}}\right] \right) ^{p-c} 
	\sum_{\nu =0}^{c}\sum_{j=0}^{v}\left\langle \begin{matrix}
		v\\
		j
	\end{matrix}\right\rangle
	\left\{\begin{matrix}
		c\\
		v
	\end{matrix}\right\} B_{j}.
\end{eqnarray*}
By applying the orthogonality relation of the Stirling numbers of the first
and second kinds to the above equation, we arrive at Eq. (\ref{o1}).
\end{remark}

\begin{theorem}
Let $\ell _{1},\ell _{2},p\in \mathbb{N}_{0}$. Let $(\ell _{1},\ell _{2})=1$%
. Then%
\begin{equation*}
S_{p}(\ell _{1},\ell _{2})=\frac{1}{2^{p}}\sum_{j=0}^{p}\binom{p}{j}%
E_{p-j}S_{j}(\ell _{1},\ell _{2})
\end{equation*}
\end{theorem}

\begin{proof}
By using (\ref{eq2}), we get%
\begin{equation*}
\sum_{p=0}^{\infty }S_{p}(\ell _{1},\ell _{2})\frac{w^{p}}{p!}%
=\sum_{p=0}^{\infty }\frac{E_{p}}{2^{p}}\frac{w^{p}}{p!}\sum_{p=0}^{\infty }%
\frac{S_{p}(\ell _{1},\ell _{2})}{2^{p}}\frac{w^{p}}{p!}.
\end{equation*}%
Therefore%
\begin{equation*}
\sum_{p=0}^{\infty }S_{p}(\ell _{1},\ell _{2})\frac{w^{p}}{p!}%
=\sum_{p=0}^{\infty }\sum_{j=0}^{p}\binom{p}{j}\frac{1}{2^{p}}%
E_{p-j}S_{j}(\ell _{1},\ell _{2})\frac{w^{p}}{p!}.
\end{equation*}%
Comparing coefficient $\dfrac{w^{p}}{p!}$ on the both sides of the above
equation, we get the desired result.
\end{proof}

\begin{theorem}
Let $\ell _{1},\ell _{2},p\in \mathbb{N}_{0}$. Let $(\ell _{1},\ell _{2})=1$%
. Then%
\begin{equation*}
S_{p}(\ell _{1},\ell _{2})=\frac{1}{2^{p}\ell _{2}}\sum_{m=1}^{\ell
_{2}-1}m\sum_{j=0}^{p}\binom{p}{j}B_{p-j}E_{j}\left( \frac{2\ell _{1}m}{\ell
_{2}}-2\left[ \frac{\ell _{1}m}{\ell _{2}}\right] \right)
\end{equation*}
\end{theorem}

\begin{proof}
By using (\ref{eq3}), we get%
\begin{equation*}
\sum_{p=0}^{\infty }S_{p}(\ell _{1},\ell _{2})\frac{w^{p}}{p!}%
=\sum_{p=0}^{\infty }\frac{B_{p}}{2^{p}}\frac{w^{p}}{p!}\sum_{m=1}^{\ell
_{2}-1}m\sum_{p=0}^{\infty }\frac{E_{p}\left( \frac{2\ell _{1}m}{\ell _{2}}-2%
\left[ \frac{\ell _{1}m}{\ell _{2}}\right] \right) }{2^{p}}\frac{w^{p}}{p!}.
\end{equation*}%
Therefore%
\begin{equation*}
\sum_{p=0}^{\infty }S_{p}(\ell _{1},\ell _{2})\frac{w^{p}}{p!}=\frac{1}{\ell
_{2}}\sum_{m=1}^{\ell _{2}-1}m\sum_{p=0}^{\infty }\sum_{j=0}^{p}\binom{p}{j}%
\frac{1}{2^{p}}B_{p-j}E_{j}\left( \frac{2\ell _{1}m}{\ell _{2}}-2\left[ 
\frac{\ell _{1}m}{\ell _{2}}\right] \right) \frac{w^{p}}{p!}.
\end{equation*}%
Comparing coefficient $\dfrac{w^{p}}{p!}$ on the both sides of the above
equation, we get the desired result.
\end{proof}

\begin{theorem}
Let $\ell _{1},\ell _{2},p\in \mathbb{N}_{0}$. Let $(\ell _{1},\ell _{2})=1$%
. Then%
\begin{equation*}
S_{p}(\ell _{1},\ell _{2})=\frac{1}{2^{p}}\sum_{m=1}^{\ell
_{2}-1}m\sum_{j=0}^{p}\binom{p}{j}E_{p-j}B_{j}\left( \frac{2\ell _{1}m}{\ell
_{2}}-2\left[ \frac{\ell _{1}m}{\ell _{2}}\right] \right) .
\end{equation*}
\end{theorem}

\begin{proof}
By using (\ref{eg1}), we get%
\begin{equation*}
\sum_{p=0}^{\infty }S_{p}(\ell _{1},\ell _{2})\frac{w^{p}}{p!}=\frac{1}{\ell
_{2}}\sum_{p=0}^{\infty }\frac{E_{p}}{2^{p}}\frac{w^{p}}{p!}\sum_{m=1}^{\ell
_{2}-1}m\sum_{p=0}^{\infty }\frac{B_{p}\left( \frac{2\ell _{1}m}{\ell _{2}}-2%
\left[ \frac{\ell _{1}m}{\ell _{2}}\right] \right) }{2^{p}}\frac{w^{p}}{p!}.
\end{equation*}%
Therefore%
\begin{equation*}
\sum_{p=0}^{\infty }S_{p}(\ell _{1},\ell _{2})\frac{w^{p}}{p!}=\frac{1}{%
2^{p}\ell _{2}}\sum_{m=1}^{\ell _{2}-1}m\sum_{p=0}^{\infty }\sum_{j=0}^{p}%
\binom{p}{j}E_{p-j}B_{j}\left( \frac{2\ell _{1}m}{\ell _{2}}-2\left[ \frac{%
\ell _{1}m}{\ell _{2}}\right] \right) \frac{w^{p}}{p!}.
\end{equation*}%
Comparing coefficient $\dfrac{w^{p}}{p!}$ on the both sides of the above
equation, we get the desired result.
\end{proof}

\begin{theorem}
Let $\ell _{1},\ell _{2},p\in \mathbb{N}_{0}$. Let $(\ell _{1},\ell _{2})=1$%
. Then%
\begin{equation*}
S_{p}(\ell _{1},\ell _{2})=\frac{1}{2^{p}\ell _{2}}\sum_{m=1}^{\ell
_{2}-1}\sum_{j=0}^{p}\binom{p}{j}mE_{p-j}\left( \frac{\ell _{1}m}{\ell _{2}}-%
\left[ \frac{\ell _{1}m}{\ell _{2}}\right] \right) B_{j}\left( \frac{\ell
_{1}m}{\ell _{2}}-\left[ \frac{\ell _{1}m}{\ell _{2}}\right] \right) .
\end{equation*}
\end{theorem}

\begin{proof}
By using (\ref{eg}), we get%
\begin{eqnarray*}
\sum_{p=0}^{\infty }S_{p}(\ell _{1},\ell _{2})\frac{w^{p}}{p!} &=&\frac{1}{%
\ell _{2}}\sum_{m=1}^{\ell _{2}-1}m\sum_{p=0}^{\infty }\frac{E_{p}\left( 
\frac{\ell _{1}m}{\ell _{2}}-\left[ \frac{\ell _{1}m}{\ell _{2}}\right]
\right) }{2^{p}}\frac{w^{p}}{p!} \\
&&\times \sum_{p=0}^{\infty }\frac{B_{p}\left( \frac{\ell _{1}m}{\ell _{2}}-%
\left[ \frac{\ell _{1}m}{\ell _{2}}\right] \right) }{2^{p}}\frac{w^{p}}{p!}.
\end{eqnarray*}%
Therefore%
\begin{eqnarray*}
\sum_{p=0}^{\infty }S_{p}(\ell _{1},\ell _{2})\frac{w^{p}}{p!} &=&\frac{1}{%
\ell _{2}}\sum_{m=1}^{\ell _{2}-1}m\sum_{p=0}^{\infty }\sum_{j=0}^{p}\binom{p%
}{j}\frac{1}{2^{p}}E_{p-j}\left( \frac{\ell _{1}m}{\ell _{2}}-\left[ \frac{%
\ell _{1}m}{\ell _{2}}\right] \right) \\
&&\times B_{j}\left( \frac{\ell _{1}m}{\ell _{2}}-\left[ \frac{\ell _{1}m}{%
\ell _{2}}\right] \right) \frac{w^{p}}{p!}.
\end{eqnarray*}%
Comparing coefficient $\dfrac{w^{p}}{p!}$ on the both sides of the above
equation, we get the desired result.
\end{proof}

\subsection{Generating function for the Apostol-Dedekind sum of order $k$}

In this section, we construct generating function for the Apostol-Dedekind
sum of order $k$ ($k\in \mathbb{N}_{0}$). We give many new formulas and
relations for the sums $\mathcal{A}_{p}\left( \ell _{1},\ell _{2}\right) $
of order $k$ and the Apostol-Dedekind sum of order $k$.

By using Eq. (\ref{A}), we set%
\begin{equation}
H^{k}(w;\left( \ell _{1},\ell _{2}\right) )=\left( g(w,0)F(w,\left( \ell
_{1},\ell _{2}\right) )\right) ^{k},  \label{Gad}
\end{equation}%
where%
\begin{equation*}
H^{k}(w;\left( \ell _{1},\ell _{2}\right) )=\sum\limits_{p=0}^{\infty
}S_{p}^{(k)}(\ell _{1},\ell _{2})\frac{w^{p}}{p!}
\end{equation*}%
and%
\begin{equation*}
F(w,\left( \ell _{1},\ell _{2}\right) )^{k}=\sum_{p=0}^{\infty }\mathcal{A}%
_{p}^{(k)}\left( \ell _{1},\ell _{2}\right) \frac{w^{p}}{p!}.
\end{equation*}%
Here note that%
\begin{equation*}
S_{p}(\ell _{1},\ell _{2}):=S_{p}^{(1)}(\ell _{1},\ell _{2})
\end{equation*}%
and%
\begin{equation*}
\mathcal{A}_{p}\left( \ell _{1},\ell _{2}\right) :=\mathcal{A}%
_{p}^{(1)}\left( \ell _{1},\ell _{2}\right) .
\end{equation*}

By (\ref{Gad}), we give some computational formulas for the Apostol-Dedekind
sum of order $k$, $S_{p}^{(k)}(\ell _{1},\ell _{2})$ as follows:

\begin{theorem}
Let $k,\ell _{1},\ell _{2},p\in \mathbb{N}_{0}$. Let $(\ell _{1},\ell
_{2})=1 $. Then%
\begin{equation}
S_{p}^{(k)}(\ell _{1},\ell _{2})=\sum\limits_{j=0}^{p}\binom{p}{j}%
B_{p-j}^{(k)}\mathcal{A}_{j}^{(k)}\left( \ell _{1},\ell _{2}\right)
\label{ah3}
\end{equation}
\end{theorem}

\begin{proof}
By using (\ref{Gad}), we obtain%
\begin{equation*}
\sum\limits_{p=0}^{\infty }S_{p}^{(k)}(\ell _{1},\ell _{2})\frac{w^{p}}{p!}%
=\sum_{p=0}^{\infty }B_{p}^{(k)}\frac{w^{p}}{p!}\sum_{p=0}^{\infty }\mathcal{%
A}_{p}^{(k)}\left( \ell _{1},\ell _{2}\right) \frac{w^{p}}{p!}.
\end{equation*}%
By using the Cauchy product, we get%
\begin{equation*}
\sum\limits_{p=0}^{\infty }S_{p}^{(k)}(\ell _{1},\ell _{2})\frac{w^{p}}{p!}%
=\sum_{p=0}^{\infty }\sum\limits_{j=0}^{p}\binom{p}{j}B_{p-j}^{(k)}\mathcal{A%
}_{j}^{(k)}\left( \ell _{1},\ell _{2}\right) \frac{w^{p}}{p!}.
\end{equation*}%
Comparing coefficient $\dfrac{w^{p}}{p!}$ on the both sides of the above
equation, we get the desired result.
\end{proof}

Since%
\begin{equation*}
H^{k_{1}+k_{2}}(w;\left( \ell _{1},\ell _{2}\right) )=H^{k_{1}}(w;\left(
\ell _{1},\ell _{2}\right) )H^{k_{2}}(w;\left( \ell _{1},\ell _{2}\right) ).
\end{equation*}%
Similarly, this equation yields the following result:

\begin{corollary}
Let $k_{1},k_{2},\ell _{1},\ell _{2},p\in \mathbb{N}_{0}$ and $\left( \ell
_{1},\ell _{2}\right) =1$. Then%
\begin{equation}
S_{p}^{(k_{1}+k_{2})}(\ell _{1},\ell _{2})=\sum\limits_{j=0}^{p}\binom{p}{j}%
S_{j}^{(k_{1})}(\ell _{1},\ell _{2})S_{p-j}^{(k_{2})}(\ell _{1},\ell _{2}).
\label{ah2}
\end{equation}
\end{corollary}

For $k=1$, we have%
\begin{equation*}
S_{p}(\ell _{1},\ell _{2})=\sum\limits_{j=0}^{p}\binom{p}{j}B_{p-j}\mathcal{A%
}_{j}\left( \ell _{1},\ell _{2}\right) .
\end{equation*}%
Combining this equation with (\ref{ah7}), we show that the Apostol-Dedekind
sums $S_{p}(\ell _{1},\ell _{2})$ can be given in terms of the polynomials $%
Q_{j,c}(x,y;\ell _{1},\ell _{2})$ by following corollary:

\begin{corollary}
Let $\ell _{1},\ell _{2},p\in \mathbb{N}_{0}$ and $(\ell _{1},\ell _{2})=1$.
Then%
\begin{equation*}
S_{p}(\ell _{1},\ell _{2})=\sum\limits_{j=0}^{p}\binom{p}{j}%
B_{p-j}\sum_{c=1}^{j}(-1)^{j-c}\binom{j}{c}\frac{\ell _{1}^{j-c}}{\ell
_{2}^{j-c+1}}Q_{j+1-c,c}(1,1;\ell _{1},\ell _{2}).
\end{equation*}
\end{corollary}

Puttin $k=2$ in (\ref{ah3}) and (\ref{ah2}), we get
\begin{equation*}
S_{p}^{(2)}(\ell _{1},\ell _{2})=\sum\limits_{j=0}^{p}\binom{p}{j}S_{j}(\ell
_{1},\ell _{2})S_{p-j}(\ell _{1},\ell _{2}),
\end{equation*}%
and%
\begin{equation*}
S_{p}^{(2)}(\ell _{1},\ell _{2})=\sum\limits_{j=0}^{p}\binom{p}{j}%
B_{p-j}^{(2)}\mathcal{A}_{p}^{(2)}\left( \ell _{1},\ell _{2}\right) .
\end{equation*}

In order to give computational formula for $\mathcal{A}_{p}^{(k)}\left( \ell
_{1},\ell _{2}\right) $, we set%
\begin{equation}
H^{k}(w;\left( \ell _{1},\ell _{2}\right) )g^{-k}(w,0)=F^{k}(w,\left( \ell
_{1},\ell _{2}\right) ),  \label{g1}
\end{equation}%
where the function $g^{-k}(w,0)$ is a generating function for the Bernoulli numbers of negative order, $B_{p-j}^{(-k)}$ and%
\begin{equation}
F^{k}(w,\left( \ell _{1},\ell _{2}\right) =\frac{k!}{t^{k}}H^{k}(w;\left(
\ell _{1},\ell _{2}\right) )Fs2(w,k).  \label{g2}
\end{equation}

\begin{theorem}
Let $k,\ell _{1},\ell _{2},p\in \mathbb{N}_{0}$ and $(\ell _{1},\ell _{2})=1$%
. Then%
\begin{equation}
\mathcal{A}_{p}^{(k)}\left( \ell _{1},\ell _{2}\right) =\sum_{j=0}^{p}\binom{%
p}{j}S_{j}^{(k)}(\ell _{1},\ell _{2})B_{p-j}^{(-k)}.  \label{g1a}
\end{equation}
\end{theorem}

\begin{proof}
By using (\ref{g1}), we get%
\begin{equation*}
\sum_{p=0}^{\infty }\mathcal{A}_{p}^{(k)}\left( \ell _{1},\ell _{2}\right) 
\frac{w^{p}}{p!}=\sum\limits_{p=0}^{\infty }S_{p}^{(k)}(\ell _{1},\ell _{2})%
\frac{w^{p}}{p!}\sum_{p=0}^{\infty }B_{p}^{(-k)}\frac{w^{p}}{p!}.
\end{equation*}%
Therefore%
\begin{equation*}
\sum_{p=0}^{\infty }\mathcal{A}_{p}^{(k)}\left( \ell _{1},\ell _{2}\right) 
\frac{w^{p}}{p!}=\sum\limits_{p=0}^{\infty }\sum_{j=0}^{p}\binom{p}{j}%
S_{j}^{(k)}(\ell _{1},\ell _{2})B_{p-j}^{(-k)}\frac{w^{p}}{p!}.
\end{equation*}%
Comparing coefficient $\dfrac{w^{p}}{p!}$ on the both sides of the above
equation, we get the desired result.
\end{proof}

In order to find values of $S_{p}^{(k)}(\ell _{1},\ell _{2})$ and $\mathcal{A%
}_{p}^{(k)}\left( \ell _{1},\ell _{2}\right) $, we also need the following
formulas.

Substituting $k=k_{1}+k_{2}$ into (\ref{g1}), we also get the following
result:

\begin{corollary}
Let $k_{1},k_{2},\ell _{1},\ell _{2},p\in \mathbb{N}_{0}$ and $\left( \ell
_{1},\ell _{2}\right) =1$. Then%
\begin{equation}
\mathcal{A}_{p}^{(k_{1}+k_{2})}\left( \ell _{1},\ell _{2}\right)
=\sum_{j=0}^{p}\binom{p}{j}\mathcal{A}_{j}^{(k_{1})}\left( \ell _{1},\ell
_{2}\right) \mathcal{A}_{p-j}^{(k_{2})}\left( \ell _{1},\ell _{2}\right) .
\label{ah1}
\end{equation}
\end{corollary}

For $p=0,1$, Eq. (\ref{ah1}) gives us%
\begin{equation*}
\mathcal{A}_{0}^{(k_{1}+k_{2})}\left( \ell _{1},\ell _{2}\right) =\mathcal{A}%
_{0}^{(k_{1})}\left( \ell _{1},\ell _{2}\right) \mathcal{A}%
_{0}^{(k_{2})}\left( \ell _{1},\ell _{2}\right) ,
\end{equation*}%
and%
\begin{equation*}
\mathcal{A}_{1}^{(k_{1}+k_{2})}\left( \ell _{1},\ell _{2}\right) =\mathcal{A}%
_{0}^{(k_{1})}\left( \ell _{1},\ell _{2}\right) \mathcal{A}%
_{1}^{(k_{2})}\left( \ell _{1},\ell _{2}\right) +\mathcal{A}%
_{1}^{(k_{1})}\left( \ell _{1},\ell _{2}\right) \mathcal{A}%
_{0}^{(k_{2})}\left( \ell _{1},\ell _{2}\right) .
\end{equation*}%
For $k_{1}=k_{2}$, we have%
\begin{equation*}
\mathcal{A}_{0}^{(2k_{1})}\left( \ell _{1},\ell _{2}\right) =\left( \mathcal{%
A}_{0}^{(k_{1})}\left( \ell _{1},\ell _{2}\right) \right) ^{2},
\end{equation*}%
and%
\begin{equation*}
\mathcal{A}_{1}^{(2k_{1})}\left( \ell _{1},\ell _{2}\right) =2\mathcal{A}%
_{0}^{(k_{1})}\left( \ell _{1},\ell _{2}\right) \mathcal{A}%
_{1}^{(k_{2})}\left( \ell _{1},\ell _{2}\right) .
\end{equation*}

\begin{theorem}
Let $k,\ell _{1},\ell _{2},p\in \mathbb{N}_{0}$ and $\left( \ell _{1},\ell
_{2}\right) =1$. Then%
\begin{equation}
\mathcal{A}_{p}^{(k)}\left( \ell _{1},\ell _{2}\right) =\frac{1}{\binom{p}{k}%
}\sum_{j=0}^{p+k}\binom{p+k}{j}
\left\langle
\begin{matrix}
	p+k-j\\
	k
\end{matrix}\right\rangle
S_{j}^{(k)}(\ell _{1},\ell _{2}).  \label{g2a}
\end{equation}
\end{theorem}

\begin{proof}
By using (\ref{g2}), we get%
\begin{equation*}
\sum_{p=0}^{\infty }\mathcal{A}_{p}^{(k)}\left( \ell _{1},\ell _{2}\right) 
\frac{w^{p}}{p!}=\frac{k!}{t^{k}}\sum\limits_{p=0}^{\infty }S_{p}^{(k)}(\ell
_{1},\ell _{2})\frac{w^{p}}{p!}\sum_{p=0}^{\infty }\left\langle
\begin{matrix}
	p\\
	k
\end{matrix}\right\rangle \frac{w^{p}}{p!}.
\end{equation*}%
Therefore%
\begin{equation*}
\sum_{p=0}^{\infty }\mathcal{A}_{p}^{(k)}\left( \ell _{1},\ell _{2}\right) 
\frac{w^{p}}{p!}=k!\sum\limits_{p=0}^{\infty }\sum_{j=0}^{p}\binom{p}{j}%
\left\langle
\begin{matrix}
	p-j\\
	k
\end{matrix}\right\rangle S_{j}^{(k)}(\ell _{1},\ell _{2})\frac{%
w^{p-k}}{p!}.
\end{equation*}%
Comparing coefficient $\dfrac{w^{p}}{p!}$ on the both sides of the above
equation, we get the desired result.
\end{proof}

We now give some special values for the sums $S_{k}^{(k)}(\ell _{1},\ell
_{2})$ and $\mathcal{A}_{p}^{(k)}\left( \ell _{1},\ell _{2}\right) $.

Putting $k=2$, and $p=1$ in (\ref{ah2}), the Dedekind sums of order $2$ can
be given by in terms of the Dedekind sums:%
\begin{equation*}
S_{1}^{(2)}(\ell _{1},\ell _{2})=\left( \ell _{2}-1\right) s(\ell _{1},\ell
_{2}).
\end{equation*}%
Putting $k=2$, and $p=1$ in (\ref{ah3}) and using (\ref{ah2}), we get%
\begin{equation*}
S_{1}^{(2)}(\ell _{1},\ell _{2})=B_{1}^{(2)}\mathcal{A}_{0}^{(2)}\left( \ell
_{1},\ell _{2}\right) +B_{0}^{(2)}\mathcal{A}_{1}^{(2)}\left( \ell _{1},\ell
_{2}\right) .
\end{equation*}%
Since $B_{1}^{(2)}=-1$, $B_{0}^{(2)}=1$ and $\mathcal{A}_{0}^{(2)}\left(
\ell _{1},\ell _{2}\right) =\left( \mathcal{A}_{0}\left( \ell _{1},\ell
_{2}\right) \right) ^{2}$, we get%
\begin{eqnarray*}
S_{1}^{(2)}(\ell _{1},\ell _{2}) &=&-\left( \mathcal{A}_{0}\left( \ell
_{1},\ell _{2}\right) \right) ^{2}+2\mathcal{A}_{0}\left( \ell _{1},\ell
_{2}\right) \mathcal{A}_{1}\left( \ell _{1},\ell _{2}\right) \\
&=&\frac{\ell _{2}-1}{2}\left( 2s\left( \ell _{1},\ell _{2}\right) -\frac{%
1-\ell _{2}}{2}-\frac{\ell _{2}-1}{2}\right) \\
&=&\left( \ell _{2}-1\right) s(\ell _{1},\ell _{2}).
\end{eqnarray*}%
Putting $k=2$, and $p=2$ in (\ref{ah3}), the Dedekind sums of order $2$ is
given by in terms of the Dedekind sums:%
\begin{eqnarray*}
S_{2}^{(2)}(\ell _{1},\ell _{2}) &=&\binom{2}{0}B_{2}^{(2)}\mathcal{A}%
_{0}^{(2)}\left( \ell _{1},\ell _{2}\right) +\binom{2}{1}B_{1}^{(2)}\mathcal{%
A}_{1}^{(2)}\left( \ell _{1},\ell _{2}\right) +\binom{2}{2}B_{0}^{(2)}%
\mathcal{A}_{2}^{(2)}\left( \ell _{1},\ell _{2}\right) \\
&=&\frac{1}{6}\mathcal{A}_{0}^{(2)}\left( \ell _{1},\ell _{2}\right) -2%
\mathcal{A}_{1}^{(2)}\left( \ell _{1},\ell _{2}\right) +\mathcal{A}%
_{2}^{(2)}\left( \ell _{1},\ell _{2}\right)
\end{eqnarray*}%
By using (\ref{ah1}), we get%
\begin{eqnarray*}
S_{2}^{(2)}(\ell _{1},\ell _{2}) &=&\frac{1}{6}\mathcal{A}_{0}^{(2)}\left(
\ell _{1},\ell _{2}\right) -2\mathcal{A}_{1}^{(2)}\left( \ell _{1},\ell
_{2}\right) \\
&&+2\mathcal{A}_{0}\left( \ell _{1},\ell _{2}\right) \mathcal{A}_{2}\left(
\ell _{1},\ell _{2}\right) +2\left( \mathcal{A}_{1}\left( \ell _{1},\ell
_{2}\right) \right) ^{2}.
\end{eqnarray*}%
The above equation related to Dedekind sums, and the polynomials $Q_{v,\ell
}(x,y;\ell _{1},\ell _{2})$.

A reciprocity law of the sums  $S_{p}^{(k)}\left( \ell _{1},\ell _{2}\right) $ is given by the aid of the sums  $\mathcal{A}_{p}^{(k)}\left( \ell _{1},\ell _{2}\right)$ and the Bernoulli numbers of order $k$ can be given. That is,  
\begin{equation*}
	S_{p}^{(k)}\left( \ell _{1},\ell _{2}\right)+S_{p}^{(k)}\left( \ell _{2},\ell _{1}\right)=\sum\limits_{j=0}^{p}\binom{p}{j}%
	B_{p-j}^{(k)}\left(	\mathcal{A}_{j}^{(k)}\left( \ell _{1},\ell _{2}\right)+\mathcal{A}_{j}^{(k)}\left( \ell _{2},\ell _{1}\right)\right),
\end{equation*}
where $( \ell _{1},\ell _{2})=1.$


\begin{thebibliography}{99}
	\bibitem{Apostol1} T.M. Apostol, Generalized Dedekind sums and
transformation formulae of certain Lambert series, Duke Math. J.  17 (1950) 	147--157. 

\bibitem{Apostol} T.M. Apostol, Theorems on generalized Dedekind sums,
Pacific J. Math. 2 (1952) 1--9.

\bibitem{AV} 	T. M. Apostol, T.H. Vu, Elementary proofs of Berndt’s reciprocity laws, Pac. J. Math. 98 (1982) 17--23 

\bibitem{Bayad} A. Bayad, Jacobi forms in two variables: Multiple elliptic
Dedekind sums, the Kummer-Von Staudt Clausen congruences for elliptic
Bernoulli functions and values of Hecke $L$-functions, Montes Taurus J.
Pure Appl. Math. 1 (2) (2019) 58--129.



\bibitem{Beck} M. Beck, Geometric proofs of polynomial reciprocity laws of
Carlitz, Berndt, and Dieter. In: Diophantine Analysis and Related Fields.
Sem. Math. Sci. 35  (2006) 11--18.

\bibitem{BeckS} M. Beck, F. Kohl, Rademacher-Carlitz polynomials,
https://arxiv.org/pdf/1310.0380, 2013.

\bibitem{BeckArX} M. Beck, C. Haase,  A.R. Matthews, Dedekind--Carlitz
Polynomials as Lattice-Point Enumerators in Rational Polyhedra,
https://arxiv.org/html/0710.1323v2, 2008.

\bibitem{BerndtDieter} B.C. Berndt, U. Dieter, Sums involving the
greatest integer function and Riemann Stieltjes integration, J. Reine Angew.
Math. 337 (1982) 208--220. 

\bibitem{BerndtGoldberg} B.C. Berndt, L.A. Goldberg, Analytic
properties of arithmetic sums arising in the theory of the classical
Theta-functions, SIAM J. Math. Anal. 15 (1984) 143--150.

\bibitem{Broven} S. Brown, On a family of sums of powers of the floor
function and their links with generalized Dedekind sums, Notes Number Theory Discret Math. 32 (1) (2026) 76--87, DOI:
10.7546/nntdm.2026.32.1.76-87; https://arxiv.org/pdf/2507.11666.

\bibitem{Carlitz} L. Carlitz, Some polynomials associated with Dedekind
sums, Acta Math. Acad. Sci. Hung. 26 (1975) 311--319.

\bibitem{Cetin} E. Cetin, I.N. Cangul, Y. Simsek, Some special finite
sums related to the three-term polynomial relations and their applications.
Adv. Difference Equ. 2014 (283) (2014),  DOI:
10.1186/1687-1847-2014-283.

\bibitem{Gessel} I. Gessel, Generating functions and generalized Dedekind
sums, 	Electron. J. Comb.  4 (2)  \#R11 (1996) 1--17.
https://doi.org/10.37236/1326

\bibitem{gradimir} G.V. Milovanovi\'{c}, Y. Simsek, Dedekind and Hardy type
sums and trigonometric sums induced by quadrature formulas, in:
Trigonometric Sums and Their
Applications (Ed. by A. Raigorodskii, M. Rassias), Springer, Cham., (2020) 183--228, 
https://doi.org/10.1007/978-3-030-37904-9\_10.

\bibitem{Hu} S. Hu, D. Kim, M.-S. Kim, On reciprocity formula of
Apostol-Dedekind sum with quasi-periodic Euler functions, J. Number Theory
162 (2016) 54--67.

\bibitem{Mattew} A.R. Matthews, A Geometric Approach to Carlitz-Dedekind
Sums, A thesis presented to the faculty of San Francisco State University in
partial fulfilment of The requirements for the degree Master of Arts in
Mathematics, San Francisco, California May, 2007.

\bibitem{Milne} L.M. Milne-Thomson, Calculus of Finite Differences, American
Mathematical Society, Chelsea, 1980.

\bibitem{PS}  M. R.	Pettet, R. Sitaramachandrarao, Three-term relations for Hardy sums, J. Number Theory 25 (1989) 328--339 

\bibitem{RW} H. Rademacher, A. Whiteman, Theorems on Dedekind sums, Amer. J. math. 63 (1941) 377--407 
\bibitem{Riardon} J. Riordan, Introduction to Combinatorial Analysis,
Princeton University Press, Princeton, 1958.

\bibitem{SimsekJNT} Y. Simsek, Relations between theta-functions Hardy sums
Eisenstein series and Lambert series in the transformation formula of $\log
\eta _{g,h}(z)$, J. Number Theory 99 (2003) 338--360.

\bibitem{ascm2004} Y. Simsek, Generalized Dedekind sums associated with the
Abel sum and the Eisenstein and Lambert series. Adv. Stud. Contemp. Math.
(Kyungshang) 9 (2) (2004) 125--137.

\bibitem{SimsekJMAA} Y. Simsek, $q$-Dedekind type sums related to $q$-zeta
function and basic $L$-series, J. Math. Anal. Appl. 318 (1) (2006) 333--351.



\bibitem{ascm2006} Y. Simsek, Remarks on reciprocity laws of the Dedekind
and Hard sums, Adv. Stud. Contemp. Math. (Kyungshang) 12 (2) (2006) 237--246.

\bibitem{SimsekNA} Y. Simsek, $q$-Hardy-Berndt type sums associated
with $q$-Genocchi type zeta and $q$-$l$-functions, Nonlinear Anal. 71 (2009)  e377--e395.

\bibitem{SimsekFPTA} Y. Simsek, Generating functions for generalized
Stirling type numbers, array type polynomials, Eulerian type polynomials and
their alications. Fixed Point Theory Appl. 2013 (87) (2013).

\bibitem{SimsekBOOK} Y. Simsek, Finite sums derived from Dedekind and Hardy
sums with their approximation formulas, in: Contemporary Approaches to
Approximation of Operators Approximation Theory and Mathematical Methods (Ed. by G. V. Milovanovic, V. Gupta), Birkh\"{a}user Cham., 2026.
\end{thebibliography}
\end{document}